\documentclass[12pt,reqno]{amsart}

\usepackage{graphicx}
\usepackage[caption=false]{subfig}
\usepackage{float}
\usepackage{caption}
\usepackage{amsmath,amsfonts,amssymb,amsthm}
\usepackage{mathrsfs}
\usepackage{mathtools} 
\usepackage{siunitx}
\usepackage{relsize} 

\usepackage{enumerate}
\usepackage{enumitem}

\usepackage[utf8x]{inputenc}
\usepackage[english]{babel}
\usepackage{soul}
\usepackage{xcolor}
\usepackage{titlesec}
\usepackage[margin=1in]{geometry}

\usepackage{hyperref}   
\usepackage{varioref}   
\usepackage{aliascnt}   
\usepackage{cleveref}   

\usepackage{aliascnt}

\newtheorem{theorem}{Theorem}[section]

\newaliascnt{lemma}{theorem}
\newtheorem{lemma}[lemma]{Lemma}
\aliascntresetthe{lemma}

\newaliascnt{corollary}{theorem}
\newtheorem{corollary}[corollary]{Corollary}
\aliascntresetthe{corollary}

\newaliascnt{proposition}{theorem}

\aliascntresetthe{proposition}

\theoremstyle{definition}
\newaliascnt{definition}{theorem}

\aliascntresetthe{definition}

\theoremstyle{remark}
\newaliascnt{remark}{theorem}
\newtheorem{remark}[remark]{Remark}
\aliascntresetthe{remark}

\newtheorem{conjecture}{Conjecture} 

\Crefname{theorem}{Theorem}{Theorems}
\Crefname{lemma}{Lemma}{Lemmas}
\Crefname{corollary}{Corollary}{Corollaries}
\Crefname{proposition}{Proposition}{Propositions}
\Crefname{definition}{Definition}{Definitions}
\Crefname{remark}{Remark}{Remarks}
\Crefname{example}{Example}{Examples}
\Crefname{conjecture}{Conjecture}{Conjectures}

\titleformat{\section}[block]{\normalfont\Large\bfseries}{\thesection}{1em}{}
\titlespacing*{\section}{0pt}{10pt plus 4pt minus 2pt}{7pt plus 4pt minus 2pt}

\titleformat{\subsection}[block]{\normalfont\large\bfseries}{\thesubsection}{1em}{}
\titlespacing*{\subsection}{0pt}{8pt plus 3pt minus 2pt}{5pt plus 2pt minus 1pt}

\numberwithin{equation}{section}

\begin{document}
\title[Sharp Coefficient Bounds for certain $q$-Starlike Functions]{Sharp Coefficient Bounds for certain $q$-Starlike Functions}
\author[S. Sivaprasad Kumar]{S. Sivaprasad Kumar$^1$}
 \author[Snehal Pannu]{Snehal Pannu$^2$}
 
\address{$1,2$ Department of Applied Mathematics, Delhi Technological University, Bawana Road, Delhi-110042, INDIA}

\email{spkumar@dce.ac.in}
\email{ms.snehal$\_$25phdam07@dtu.ac.in}


\begin{abstract}
Geometric function theory increasingly draws on $q$-calculus to model discrete and quantum-inspired phenomena. Motivated by this, the present paper introduces   new subclasses of analytic functions: the class $\mathcal{S}^{*}_{\xi_q}$ of $q$-starlike functions associated with the Ma-Minda function $\xi_q(z)$, and its limiting classical counterpart $\mathcal{S}^{*}_{\xi}$ associated with $\xi(z)$, where $q \in (0,1)$. We systematically establish sharp coefficient estimates including the Fekete-Szeg\"{o}, Hankel and Toeplitz determinants.  We establish the sharpness of the 
$q$-coefficient estimates using a newly derived integral representation, which offers a more effective alternative to the conventional convolution-based extremal construction. It is further shown that all 
$q$-results reduce to their classical counterparts as $q \to 1^{-}$.
\end{abstract}

\maketitle

\section{Introduction}
\label{intro}
Let \(\mathcal{A}\) denote the family of all normalized analytic functions \(f\) defined on the open unit disk \(\mathbb{D} = \{ z \in \mathbb{C} : |z| < 1 \}\) with the Taylor series expansion
\begin{align}
f(z) = z + \sum_{n=2}^{\infty} a_n z^n\,. \label{1.1}
\end{align}

Let \(\mathcal{P}\) be the class of Carathéodory functions, consisting of analytic functions \(p\) defined on \(\mathbb{D}\) of the form
\begin{align}
    p(z) = 1 + \sum_{n=1}^\infty c_n z^n\quad (z \in \mathbb{D}), \label{2.1}
\end{align}
satisfying \(\Re(p(z)) > 0\) and \(p(0)=1\). Furthermore, let \(\mathcal{B}_0\) denote the class of Schwarz functions \(w\) analytic in \(\mathbb{D}\) with the expansion
\begin{align}
    w(z)&=\sum_{n=1}^\infty b_n z^n \quad (z \in \mathbb{D}), \label{schwarz}
\end{align}
where \(w(0)=0\) and \(|w(z)| < 1\).

Let \(\mathcal{S}\) be the subclass of \(\mathcal{A}\) consisting of univalent functions. The Hadamard product (or convolution) of two functions \(f, g \in \mathcal{A}\), where $f$ is given by \eqref{1.1} and \(g(z) = z + \sum_{n=2}^\infty d_n z^n\), is defined as
\[
(f * g)(z) = z + \sum_{n=2}^\infty a_n d_n z^n.
\]
This operation provides a powerful tool for expressing linear operators; for instance, the derivative can be written as
\begin{align}
f'(z) = \frac{1}{z} \left( f(z) * \frac{z}{(1-z)^2} \right). \notag
\end{align}

Recently, Piejko \emph{et al}. \cite{15} introduced a generalized operator defined by
\begin{align}
    d_\eta f(z) = \frac{1}{z} \left( f(z) * \frac{z}{(1-\eta z)(1-z)} \right), \quad \eta \in \mathbb{C},\ |\eta| \leq 1. \label{df'}
\end{align}
This operator generalizes fundamental concepts in calculus. For \(\eta = 1\), it reduces to the standard derivative \(f'\). When \(\eta = q\) is a real number with \(0 < q < 1\), it yields the Jackson \(q\)-derivative:
\[
d_q f(z) =
\begin{cases}
\dfrac{f(z) - f(qz)}{(1-q)z}, & z \ne 0, \\
f'(0), & z = 0,
\end{cases}
\]
with the series representation \(d_q f(z) = \sum_{n=1}^{\infty} [n]_q a_n z^{n-1}, (a_1=1)\). Here, the \(q\)-number is given by \([n]_q = \sum_{n=0}^{n-1} q^n\) for \(n \in \mathbb{N}\). In particular, \(\lim_{q \to 1^-} d_q f(z) = f'(z)\).

The theory of 
$q$-calculus extends classical analysis by replacing conventional limit-based operators with a deformation parameter 
$q$, thereby introducing a discrete and scale-dependent framework. Since Jackson’s foundational work on 
$q$-differentiation and 
$q$-integration \cite{5,6}, this theory has evolved significantly and found applications in diverse areas such as optimal control, fractional calculus, and 
$q$-difference equations. The 
$q$-derivative operator plays a central role in the study of special functions, quantum theory, and statistical mechanics, where $q$-generalizations capture underlying quantum and non-uniform structures. Recent developments in this area can be seen in \cite{32,2,30,21,11}.

For two analytic functions \(f\) and \(g\), we say \(f\) is subordinate to \(g\), denoted by \(f \prec g\), if there exists a Schwarz function \(w(z) \in \mathcal{B}_0\) such that \(f(z) = g(w(z))\). If \(g\) is univalent in \(\mathbb{D}\), then \(f \prec g\) is equivalent to the conditions \(f(0) = g(0)\) and \(f(\mathbb{D}) \subseteq g(\mathbb{D})\).

\medskip

A fundamental subclass of \(\mathcal{S}\) is the class of starlike functions \(\mathcal{S}^*\), characterized analytically by
\[
\mathcal{S}^* = \left\{ f \in \mathcal{A} : \frac{z f'(z)}{f(z)} \prec \frac{1+z}{1-z} \right\}.
\]
Extensive research on starlike functions \cite{3,4,10,9,8,33} has established a robust theoretical foundation for their geometric and analytic properties. Ma and Minda \cite{24} unified this theory by introducing a general class:
\[
\mathcal{S}^*(\phi) = \left\{ f \in \mathcal{A} : \frac{z f'(z)}{f(z)} \prec \phi(z) \right\},
\]
where \(\phi\) is an analytic function with positive real part, \(\phi(0)=1\), \(\phi'(\mathbb{D})\) is starlike, symmetric about the real axis, and \(\phi'(0) > 0\). Numerous subclasses of starlike functions, now known as Ma-Minda classes, have been introduced by selecting specific \(\phi\) functions.

In this investigation, we consider the functions defined by
\begin{align}
\xi_q(z) &= 1 + \frac{\sin(qz)}{q(1 - qz)} \quad\text{and}\quad \xi(z) = 1 + \frac{\sin z}{1 - z}\quad (q \in (0,1),\, z \in \mathbb{D}). \notag 
\end{align}
Note that $\xi := \lim_{q \to 1^-} \xi_q$. As evidenced by ~\Cref{fig1} and~\Cref{fig2}, both $\xi_q$ and $\xi$ satisfy the criteria for Ma-Minda functions: they are analytic with positive real part, $\xi_q(0) = \xi(0) = 1$, their images are starlike with respect to 1 and symmetric about the real axis, and they have positive derivatives at the origin. In particular, for \(0<q<1\), the pole at \(z=1/q\) lies outside \(\mathbb{D}\), and hence \(\xi_q(\mathbb{D})\) is bounded. As \(q\to1^{-}\), the pole moves to \(z=1\in\partial\mathbb{D}\); consequently, \(\xi_q(\mathbb{D})\) becomes unbounded.
\begin{figure}[ht!]
    \centering
    \begin{minipage}[t]{0.48\textwidth}
        \centering
        \includegraphics[width=\textwidth]{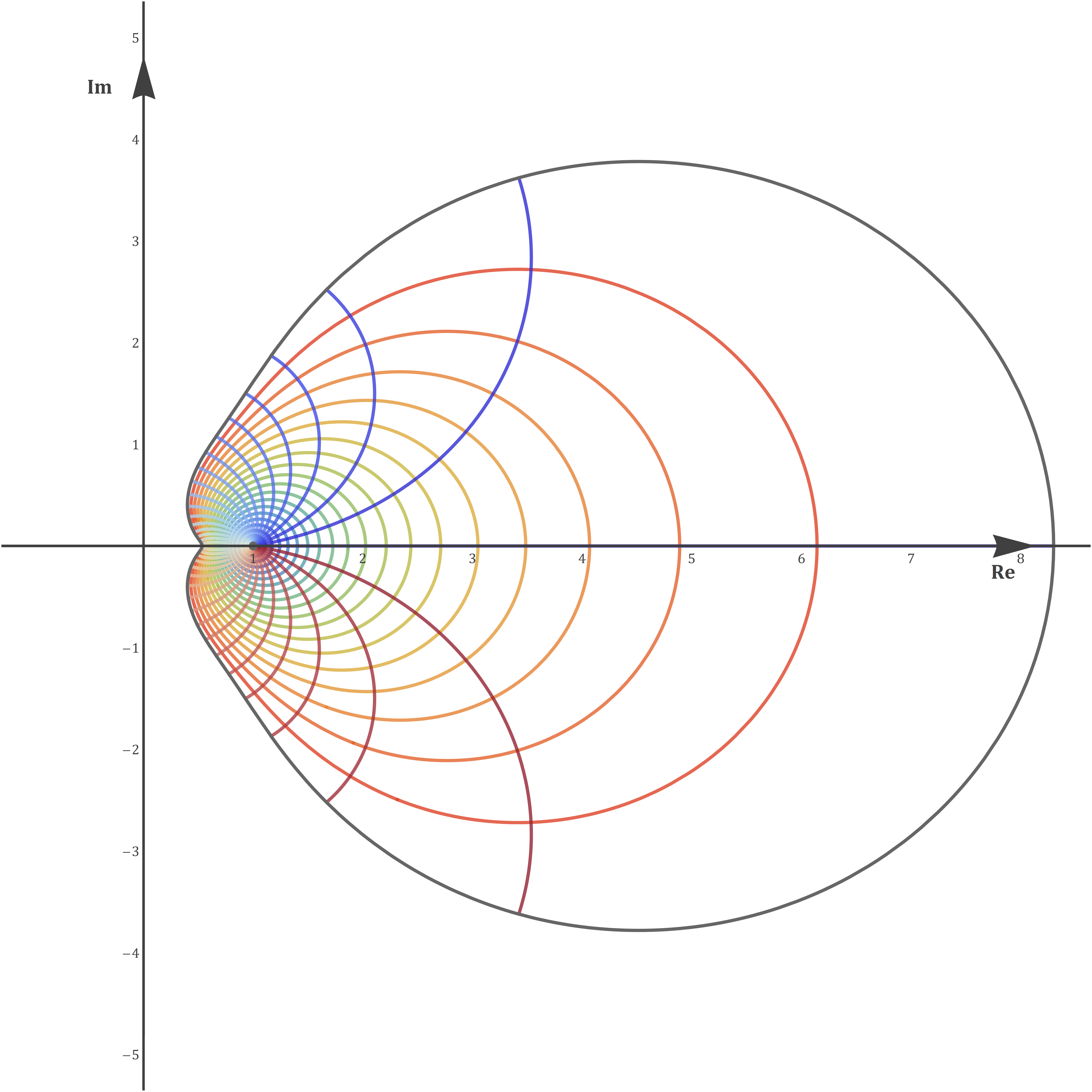}
        \caption{Image domain $\xi_{0.8}(\mathbb{D})$.}
        \label{fig1}
    \end{minipage}
    \hfill
    \begin{minipage}[t]{0.48\textwidth}
        \centering
        \includegraphics[width=\textwidth]{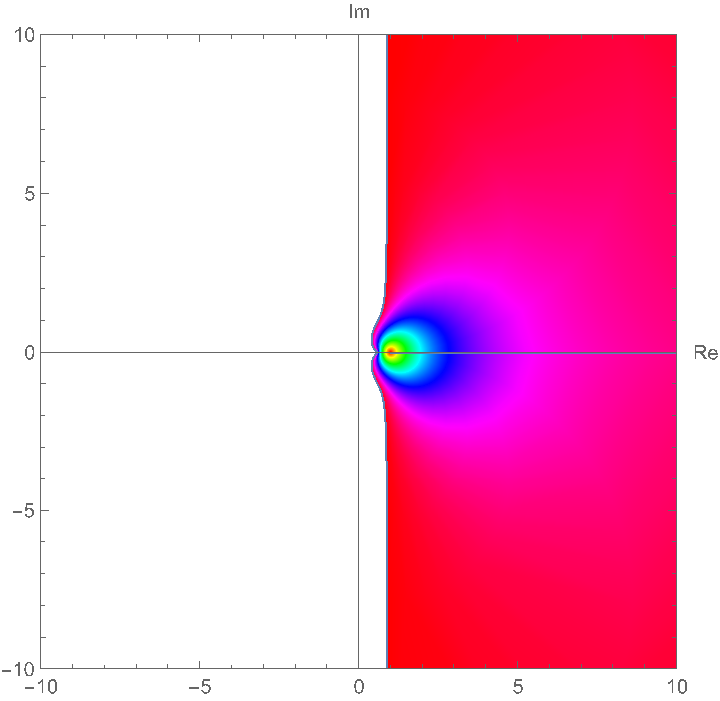}
        \caption{Image domain $\xi(\mathbb{D})$.}
        \label{fig2}
    \end{minipage}
\end{figure}

The series expansion of $\xi_q(z)$ is given by
\begin{align}
\xi_q(z) &= 1 + z + q z^2 + \frac{5}{6} q^2 z^3 + \frac{5}{6}q^3 z^4 + \frac{101}{120}q^4 z^5 + \cdots \quad (z \in \mathbb{D}), \label{phi_q_series}
\end{align}
while for $\xi(z)$ we obtain
\begin{align}
\xi(z) &= 1 + z + z^2 + \frac{5}{6}z^3 + \frac{5}{6}z^4 + \frac{101}{120}z^5 + \cdots \quad (z \in \mathbb{D}). \label{phi_series}
\end{align}

Motivated by the aforementioned Ma-Minda classes, we introduce the class of $q$-starlike functions associated with $\xi_q$:
\begin{align}
\mathcal{S}^*_{\xi_q} &= \left\{ f \in \mathcal{A} : \frac{z d_qf(z)}{f(z)} \prec \xi_q(z) \right\} \quad (z \in \mathbb{D}). \label{qstar}
\end{align}

Taking the limit as $q \to 1^-$, we obtain the corresponding class of starlike functions associated with $\xi$:
\begin{align}
\mathcal{S}^*_{\xi} &= \left\{ f \in \mathcal{A} : \frac{z f'(z)}{f(z)} \prec \xi(z) \right\} \quad (z \in \mathbb{D}). \label{star}
\end{align}

Since $\mathcal{S}^*_{\xi}$ satisfies the Ma-Minda condition corresponding to the function $\xi$, it follows that
\begin{theorem}\label{thm1.1}
    Let $f \in \mathcal{S}^*_{\xi}$. Then
    \begin{enumerate}
        \item  \textit{Subordination results}: 
    \( \frac{z\,f'(z)}{f(z)}\prec \frac{z\,   \tilde f'(z)}{\tilde f(z)} \quad\text{and}\quad \frac{f(z)}{z}\prec \frac{\tilde f(z)}{z}.\)

    \item  \textit{Growth theorem}: For $|z|=r<1$, \(
        -\tilde f(-r) \le |f(z)| \le \tilde f(r).
        \)

    \item  \textit{Distortion theorem}: For $|z|=r<1$,
        \(        -|1 - M(r)| \frac{\tilde f(-r)}{r} \le |f'(z)| \le |1 + M(r)| \frac{\tilde f(r)}{r},\) where $M(r):= \displaystyle\max_{|z|=r} \left|\frac{\sin(z)}{(1-z)}\right|$.

    \item \textit{Rotation theorem}: For $|z|=r<1$,
        \(
        \left| \arg\frac{f(z)}{z} \right| \leq \max_{|z|=r} \arg\frac{\tilde f(z)}{z}.
        \)

    \item \textit{Covering theorem}: The function $f$ is either a rotation of $\tilde f$, or its image contains the disk
        \(
        \{w \in \mathbb{C}: |w| < -\tilde f(-1)\},
        \)
        where $\tilde f(-1) = \displaystyle\lim_{r \to 1^{-}} \tilde f(-r)$.
    \end{enumerate}
\end{theorem}
In an analogous manner, the above results can be extended to the class $\mathcal{S}^*_{\xi_{q}}$.

A function \(f\) belongs to \(\mathcal{S}^*_{\xi_q}\) if and only if there exists a Schwarz function \(w(z) \in \mathcal{B}_0\) such that
\[
\frac{z d_qf(z)}{f(z)} = \xi_q(w(z)).
\]
This representation yields the integral form 
\[
f(z) = z \exp\left( \int_0^z \frac{\xi_q(w(t)) - \lambda_q}{t}  d_q t \right),
\]
where \(\lambda_q = \frac{\ln q}{q-1}\) and \(\lim_{q \to 1^-} \lambda_q = 1\).

Using the Jackson integral definition
\[
\int_0^z h(t)  d_q t = (1-q)z \sum_{k=0}^\infty q^k h(q^k z),
\]
we obtain the explicit series representation
\[
\int_0^z \frac{\xi_q(w(t)) - \lambda_q}{t}  d_q t = (1-q) \sum_{k=0}^\infty \left( \xi_q(w(q^k z)) - \lambda_q \right),
\]
provided the series converges for the given \(\xi_q\) and \(q\).

The extremal function for the class \(\mathcal{S}^*_{\xi_q}\), corresponding to \(w(z) = z\), is given by
\begin{align}
\tilde{f}_q(z) &= z \exp\left( \int_0^z \frac{\xi_q(t) - \lambda_q}{t}  d_q t \right) \nonumber \\
&= z \exp\left( \int_0^z \frac{\sin(qt) + q(1-qt)\big(1+ \frac{\ln q}{1-q}\big)}{qt(1-qt)}  d_q t \right) \in \mathcal{S}^*_{\xi_q}. \label{qsharp}
\end{align}
Its classical counterpart for \(q \to 1^-\) is
\begin{align}
\tilde{f}(z) = z \exp\left( \int_0^z \frac{\xi(t) - 1}{t}  dt \right) 
= z \exp\left( \int_0^z \frac{\sin t}{t(1-t)}  dt \right) \in \mathcal{S}^*_{\xi}. \label{xi}
\end{align}

The extremal function $\tilde{f}_q$, defined explicitly in equation \eqref{qsharp}, admits an alternative characterization through a convolution equation. Specifically, it is the unique analytic function (normalized by $\tilde{f}_q(0) = 0$ and $\tilde{f}_q'(0) = 1$) satisfying the functional relation:

\begin{align*}
    \tilde{f}_q(z) * \frac{z}{(1 - qz)(1 - z)} = \tilde{f}_q(z) \cdot \xi_q(z).
\end{align*}

A Hankel matrix is a square matrix that is symmetric about its principal diagonal, capture nonlinear 
structural behaviour supporting stable simulations. For functions \(f \in \mathcal{S}\) of the form \eqref{1.1}, Pommerenke \cite{16} defined the \(s\)th Hankel determinant as
\begin{align}
H_{s,n}(f) =
\begin{vmatrix}
a_n & a_{n+1} & \cdots & a_{n+s-1} \\
a_{n+1} & a_{n+2} & \cdots & a_{n+s} \\
\vdots & \vdots & \ddots & \vdots \\
a_{n+s-1} & a_{n+s} & \cdots & a_{n+2s-2}
\end{vmatrix}, \label{Hankel_det}
\end{align}
where \(n,s \in \mathbb{N}\) and \(a_1 = 1\). Establishing sharp upper bounds for Hankel determinants remains a central problem in geometric function theory.

Ye and Lim \cite{20} demonstrated that any \(n \times n\) matrix over \(\mathbb{C}\) can be expressed as a product of Toeplitz or Hankel matrices. Toeplitz matrices are characterized by constant entries along each diagonal and find extensive applications in quantum physics, image processing, integral equations, and signal processing. The Toeplitz determinant for \(f \in \mathcal{S}\) is defined as
\begin{align}
T_{s,n}(f) = 
\begin{vmatrix}
a_n & a_{n+1} & \dots & a_{n+s-1} \\
a_{n+1} & a_n & \dots & a_{n+s-2} \\
\vdots & \vdots & \ddots & \vdots \\
a_{n+s-1} & a_{n+s-2} & \dots & a_n
\end{vmatrix}. \label{toeplitz_det}
\end{align}

In this study of $q$-geometric function theory, coefficient bounds extend beyond classical generalizations by capturing how geometric properties of analytic functions deform under the parameter $q$. Sharp coefficient inequalities encode precise information about the image domain and, in applied contexts, about the physical or computational systems modeled by conformal maps. For example, in medical imaging, structures with spiral or helical features can be effectively modeled using $q$-starlike mappings, where incorporating $q$-coefficient constraints improves reconstruction stability in limited-data settings; refer to \cite{appn1, appn2}. From a theoretical perspective, replacing integers $n$ with $q$-integers $[n]_q$ and classical derivatives with Jackson’s $q$-difference operator fundamentally alters the classes of analytic functions and their extremal behavior. Although several studies address coefficient bounds in the $q$-calculus, sharp estimates remain relatively scarce. Motivated by this, we establish sharp bounds for initial coefficients, Hankel and Toeplitz determinants, and functionals such as Fekete--Szeg\"o for the newly introduced classes $\mathcal{S}^*_{\xi}$ and $\mathcal{S}^*_{\xi_q}$. In addition, the simultaneous study of the $q$-class and its classical counterpart highlights their structural differences while revealing the convergence of the $q$-results with the classical case.

\section{Preliminary results}\label{3}

\begin{lemma}\label{lemma1}\cite{29}
If $w(z)  \in \mathcal{B}_0$ be of the form \eqref{schwarz}, if $b_1 >0$. Then
\begin{eqnarray}
|b_1|  &\leq & 1,\nonumber\\
|b_2| & \leq & 1-|b_1|^2,\nonumber \\
|b_3| &\leq &  1-|b_1|^2-\frac{|b_2|^2}{1+|b_1|}. \nonumber
\end{eqnarray}
\end{lemma}

\begin{lemma}\cite{17} \label{lemma_2}
   If $w(z)  \in \mathcal{B}_0$ be of the form \eqref{schwarz}, if $b_1 >0$. Then
\begin{align*}
    b_2 = \alpha(1-b_1^2), \quad
    b_3 = (1-b_1^2)\left[(1-|\alpha|^2)\beta - b_1\alpha^2\right],
\end{align*}
where $\alpha, \beta \in \mathbb{C}$ with $|\alpha|, |\beta| \leq 1$.
\end{lemma}

\begin{lemma}\cite{7} \label{lemma_3}
Let $p(z)$ be of the form \eqref{2.1}, and let $\mu \in \mathbb{C}$. Then
\begin{align*}
    |c_2 - \mu c_1^2| \leq 2\max\{1,\, |2\mu-1|\}. 
\end{align*}
\end{lemma}

\begin{lemma} \cite{14} \label{lemma_4}
    If $w(z) \in \mathcal{B}_0$ be of the form \eqref{schwarz} and $\sigma, \nu \in \mathbb{R}$. Then the following sharp estimate exists.
\begin{align*}
     \left| b_3 + \sigma b_1 b_2 + \nu b_1^3 \right| \leq 
        |\nu| &\quad (\sigma,\nu) \in D_1,\qquad 
\end{align*}
where
\begin{align*}
D_1 &= \begin{cases}
\displaystyle (\sigma,\nu): |\sigma| \geq \frac{1}{2}, &  \displaystyle \nu \leq -\frac{2}{3}(|\sigma|+1), \\[10pt]
\displaystyle (\sigma,\nu): 2\leq |\sigma| \leq  4,&  \displaystyle \nu \geq \frac{1}{12}(\sigma^2+8).
\end{cases}   
\end{align*}

\end{lemma}

\begin{lemma}\label{lemma5}\cite{13}:
If $A$, $B$, $C \; \in \mathbb{R}$, let us consider
$$Y(A, B, C):= \max \{ |A + Bz + Cz^{2} | + 1 - |z|^2, \quad z \in \mathbb{\overline{D}} \}.$$

If $AC \ge 0$, then
    \[Y(A, B, C)= \displaystyle\left\{\begin{array}{ll}
        |A| + |B| + |C|,& |B| \geq 2(1  - |C|), \\
        \\
        1 + |A| +\displaystyle \frac{B^2}{4(1 - |C|)}, & |B| < 2(1 - |C|).
    \end{array} \right.\]
    
\end{lemma}
\section{\texorpdfstring{Bounds for the Classical Class $\mathcal{S}^*_{\xi}$}
                          {Bounds for the Classical Class S*xi}}

We proceed with the estimation of the sharp initial coefficient bounds.

\begin{theorem}\label{thmc1}
Let $f \in \mathcal{S}^*_{\xi}$, then 
\[
|a_2| \le 1, \quad |a_3|\le 1 \quad \text{and} \quad |a_4|\le \frac{17}{18}.
\]
These bounds are sharp.
\end{theorem}

\begin{proof}
Let $f \in \mathcal{S}^*_{\xi}$. Then by \eqref{star}, there exists a schwarz function $w(z) \in \mathcal{B}_0$ such that
\begin{align}
\frac{z\, f'(z)}{f(z)} = \xi_q\bigl(w(z)\bigr). \notag
\end{align}

Using \eqref{1.1}, we get
\begin{align}
\frac{z\, f'(z)}{f(z)} 
&= 1 + a_2 z + (-a_2^2 + 2 a_3) z^2 + (a_2^3 - 3 a_2 a_3 + 3 a_4) z^3 + \cdots. \label{4.2}
\end{align}

Similarly, using \eqref{phi_series}, we get
\begin{align}
\xi\bigl(w(z)\bigr) 
&= 1 + b_1 z + (b_1^2 + b_2) z^2 + \left(\frac{5}{6} b_1^3 + 2 b_1 b_2 + b_3\right) z^3 + \cdots. \label{4.3}
\end{align}

By comparing the coefficients of \eqref{4.2} and \eqref{4.3}, we obtain
\begin{align}
a_2 &= b_1, \label{a_2}\\
a_3 &= b_1^2 + \frac{b_2}{2}, \label{a_3}\\
a_4 &= \frac{1}{18}\bigl(17 b_1^3 + 21 b_1 b_2 + 6 b_3\bigr). \label{a_4}
\end{align}

Since $b_1 \in [0,1]$, it follows from \eqref{a_2} that $|a_2| \le 1$. Using ~\Cref{lemma_2} in \eqref{a_3}, we deduce  $|a_3| \le 1$.  
Furthermore, employing ~\Cref{lemma_4} with $\sigma = 7/2$ and $\nu = 17/6$, it follows from \eqref{a_4} that $|a_4| \le 17/18$. The sharpness of the bounds can be examined using $\tilde f$ defined in \eqref{xi}.
\end{proof}

\medskip
Next, we determine the Fekete-Szeg\"o bound for the class $\mathcal{S}^*_{\xi}$.

\begin{theorem}\label{thmc2}
Let $f \in \mathcal{S}^*_{\xi}$, then 
\[
|a_3 - \mu a_2^2| \leq \frac{1}{2}\max\left\{1,2\mu-2\right\}.
\]
\end{theorem}

\begin{proof}
Let $f \in \mathcal{S}^*_{\xi}$, then by using \eqref{a_2} and \eqref{a_3}, we obtain
\begin{align}
|a_3 - \mu a_2^2| &= \left| b_1^2 + \frac{b_2}{2} - \mu b_1^2 \right|. \label{4.6}    
\end{align}

Let $p(z) \in \mathcal{P}$. Then there exists a Schwarz function $w(z) \in \mathcal{B}_0$ such that
\begin{align}
    p(z) = \frac{1 + w(z)}{1 - w(z)} 
    \quad \implies \quad 
    w(z) = \frac{p(z) - 1}{p(z) + 1}.
    \label{2.4.16}
\end{align}

Comparing the coefficients in \eqref{2.4.16}, we obtain
\begin{align}
    2b_1 = c_1, 
    \qquad 
    4b_2 = 2c_2 - c_1^2.
    \label{2.4.17}
\end{align} 
Now, substituting \eqref{2.4.17} into \eqref{2.4.16}, and using \Cref{lemma_3}, we follow
\[
|a_3 - \mu a_2^2| 
= \left|\frac{1}{4}\left[c_2 -\Big( \frac{2\mu-1}{2}\Big) c_1^2 \right]\right| 
\le \frac{1}{2} \max\left\{1, 2\mu-2\right\}.
\]
Hence, the desired bound is established.
\end{proof}

Setting $\mu = 1$ in \Cref{thmc2}, we obtain the following sharp result:
\begin{corollary}\label{corc1}
Let $f \in \mathcal{S}^*_{\xi}$, then 
\[
|a_3 - a_2^2| \le \frac{1}{2}.
\]
\end{corollary}

The equality in the above bound is attained for the function $\tilde{f}_1: \mathbb{D} \to \mathbb{C}$, defined by
\begin{align}
\tilde f_1(z)= z\, \exp\left(\int_0^z \frac{\sin(t^2)}{t(1-t^2)}dt\right). \label{sharp2}
\end{align}

Furthermore, if $f \in \mathcal{S}^*_{\xi}$, the second Hankel determinant satisfies
\[
|H_{2,1}(f)| = |a_1 a_3 - a_2^2| \leq \frac{1}{2}, \quad \text{where } a_1 = 1.
\]

\begin{theorem}\label{thmc3}
Let $f \in \mathcal{S}^*_{\xi}$. Then 
\begin{equation}
|H_{2,2}(f)| \leq \frac{1}{4}. \label{4.5.0}
\end{equation}
The estimate is sharp.
\end{theorem}

\begin{proof}
    Let $f \in \mathcal{S}^*_{\xi}$. Then from \eqref{a_2}-\eqref{a_4}, we obtain
    \begin{align*}
       |H_{2,2}(f)|=|a_2 a_4-a_3^2|=\left|\frac{1}{36}(-2 b_1^4 + 6 b_1^2 b_3 - 9 b_3^2 + 12 b_1 b_3)\right|,
    \end{align*}
which upon substitution for $b_2$ and $b_3$ by using \Cref{lemma_2}, yields
\begin{align}
    |H_{2,2}|&= \frac{1}{36}|(-2 b_1^4 + 6 b_1^2 (1 - b_1^2) \alpha - 9 (1 - b_1^2)^2 \alpha^2 \notag\\ &+ 
  12 b_1 (1 - b_1^2) (-b_1 \alpha^2 + \beta (1 - |\alpha|^2))|. \label{4.5.1}
\end{align}
For $b_1\in \{0,1\}$, \eqref{4.5.1} reduces to
\begin{align}
    |H_{2,2}|=
    \begin{cases}
        \displaystyle\frac{|\alpha|^2}{4}\le \frac{1}{4}, \quad b_1=0,|\alpha|\le1,\\
        \\
        \displaystyle\frac{1}{18}, \quad b_1=1. \label{4.5.2}
    \end{cases}
\end{align}
For $b_1 \in (0,1)$, applying the triangle inequality to \eqref{4.5.1} and using $|\beta|\le1$, we get
\begin{align}
    |H_{2,2}|\le\frac{12b_1(1-b_1^2)}{36}Y_1(A,B,C), \label{4.5.3}
\end{align}
where
\begin{align*}
    Y_1(A,B,C)=|A+B \alpha+C \alpha^2|+1-|\alpha|^2,
\end{align*}
with
\begin{align*}
    A=-\frac{2b_1^4}{12 b_1 (1 - b_1^2)}, \quad
    B=\frac{6 b_1^2 (1 - b_1^2)}{162 b_1 (1 - b_1^2)}, \quad
    C=-\frac{9 (1 - b_1^2)^2 + b_1}{12 b_1 (1 - b_1^2)}.
\end{align*}
Since $AC \ge 0$ for $b_1 \in (0,1)$, from \Cref{lemma5} we follow
\begin{align*}
    |B|-2(1-|C|)=\frac{9 - 11 b_1 - 15 b_1^2 + 12 b_1^3 + 6 b_1^4}{6 b_1 - 6 b_1^3}.
\end{align*}
It is observed that $|B|-2(1-|C|)$ is a decreasing function in $(0,0.837669)$ and increasing function on $(0.837669,1)$, thus applying \Cref{lemma1} to \eqref{4.5.3}, we have
\begin{align}
    |H_{2,2}|\le \begin{cases}
    \displaystyle  \frac{1}{36}b_1(1-b_1^2)\bigg(1+|A|+\frac{B^2}{4(1-|C|}\bigg)=0 , \quad (0,0.837669),
    \\
    \\
      \displaystyle  \frac{1}{36}b_1(1-b_1^2)(|A|+|B|+|C|)\le \frac{1}{12}  ,\quad (0.837669,1). \label{4.5.4}
    \end{cases}
\end{align}
Now the inequality \eqref{4.5.0} can be obtained using \eqref{4.5.2} and \eqref{4.5.4}. The sharpness of the result  can be examined using $\tilde f_1$ given by \eqref{sharp2}.
\end{proof}
We now proceed to the corresponding Toeplitz determinant bounds, beginning with the second-order Toeplitz determinant:

\begin{theorem}
    If $f\in \mathcal{S}^*_\xi$, then
    \begin{align*}
        |T_{2,1}| &\le 2,\qquad
        |T_{2,2}| \le 2 \qquad\text{and}\qquad
        |T_{2,3}| \le \frac{1}{4}.
    \end{align*}
These estimates are sharp.
\end{theorem}

\begin{proof}
    If $f \in \mathcal{S}^*_{\xi}$, then by optimizing \Cref{lemma1}, we follow the bounds for $ T_{2,1}$ and $ T_{2,2}$.  Now, from \eqref{toeplitz_det}, \eqref{a_3}, \eqref{a_4} and \Cref{lemma1}, we obtain
    \begin{align*}
        |T_{2,3}(f)|&=|a_3^2-a_4^2|\\
        &=\left|\left(b_1^2 + \frac{b_2}{2}\right)^2-\frac{1}{324}\left(17 b_1^3 + 21 b_1 b_2 + 6 b_3\right)^2\right|\\
        & \le \frac{1}{4}(1 + |b_1|^2)^2-\frac{(6 + 27 |b_1| + 15 |b_1|^2 - 10 |b_1|^3 - 4 |b_1|^4 - 6 |b_2|^2)^2}{324 (1 + |b_1|)^2}
    \end{align*}
Setting $x:=|b_1|$ and $y:=|b_2|$, we obtain
\[
|T_{2,3}(f)| \le \Gamma(x,y),
\]
where
\[\Gamma(x,y)= \frac{1}{4}(1 + x^2)^2-\frac{(6 + 27 x + 15 x^2 - 10 x^3 - 4 x^4 - 6 y^2)^2}{324 (1 + y)^2}.\]
In view of \Cref{lemma1}, we seek to determine the maximum of $\Gamma$ over the admissible region $\Delta$ defined as
\begin{align}
    \Delta = \{(x, y) : 0 \le x \le 1, \, 0 \le y \le 1 - x^2\}.\label{delta}    
\end{align}
We first examine the possibility of extrema occurring at interior points of $\Delta$. Accordingly, let $(x,y)\in\Delta$. Differentiating $\Gamma$ partially with respect to $y$, we obtain

\begin{align*}
    \frac{\partial\Gamma}{\partial y}=\frac{2 y (6 + 27 x + 15 x^2 - 10 x^3 - 4 x^4 - 6 y^2)}{27 (1 + x)^2},
\end{align*}
which yields
\[
y=0 \quad \text{or} \quad y=\frac{\sqrt{6 + 27x + 15x^2 - 10x^3 - 4x^4}}{\sqrt{6}}.
\]
For the corresponding values of $y$, solving $\partial \Gamma/\partial x=0$ gives
\[
x=1 \quad \text{or} \quad x\approx 0.622.
\]
It follows that these critical points do not lie in the region $\Delta$. Consequently, we proceed to examine the behavior of $\Gamma$ on the boundary of $\Delta$, where we have:
\begin{align}
    \Gamma(x,0)&\le \frac{35}{324}, \quad (0\le x\le1),\label{case1}\\ 
    \Gamma(0,y)&\le \frac{1}{4}, \quad (0\le y\le1),\label{case2}\\
    \Gamma(x,1-x^2)&\le \frac{35}{324}, \quad (0\le x\le1).\label{case3}
\end{align}
Henceforth, the bound for $T_{2,3}$ follows at once from the above inequalities \eqref{case1}-\eqref{case3}.The sharpness of $|T_{2,2}|$ and $|T_{2,3}|$ is attained by the function $\tilde f_2 \in \mathcal{S}^*_{\xi}$, defined by
\begin{align}
    \tilde f_2(z)=z \exp \int_{0}^z\frac{\sin(it)}{t(1-it)}dt  = z+iz^2-z^3-\frac{17 i z^4}{18}+\cdots ,\label{sharp iz}
\end{align}
For $|T_{2,3}|$, the extremal function $\tilde f_1$, given by \eqref{sharp2}, is considered.   
\end{proof}

 Next, we determine the third-order Toeplitz determinant:
\begin{theorem}\label{thmc5}
If $f \in \mathcal{S}^*_{\xi}$, then 
\begin{align}
    |T_{3,2}(f)| \leq 4\qquad \text{and} \qquad 
    |T_{3,2}(f)| \leq \frac{1}{324}.\label{4.8.0}
\end{align}
The estimate is sharp.
\end{theorem}

\begin{proof}
Suppoese $f \in \mathcal{S}^*_{\xi}$, then an optimization of \Cref{lemma1} leads to bound of $T_{3,1}$. Using \eqref{toeplitz_det}, \eqref{a_2}-\eqref{a_4}, and \Cref{lemma1} with $x:=|b_1|$ and $y:=|b_2|$, we obtain
    \begin{align*}
        |T_{3,2}(f)|&=|(a_2 - a_4) (a_2^2 - 2 a_3^2 + a_2 a_4)|        \le\displaystyle\frac{\kappa_{1}\kappa_{2}}{648 (1 + b_1)^2}=:\Gamma_2(x,y)
    \end{align*}
where
\begin{align*}
    \kappa_{1}&:=6y^2-6 - 9 b_1 + 10 b_1^3 + 4 x^4 \\
    \kappa_{2}&:= 3 b_1 (1 - 4 b_2^2)+9 - 63 b_1^2 - 39 b_1^3 + 47 b_1^4 + 35 b_1^5 + 9 b_1^6 + 9 b_1^7 
\end{align*}
Thus, \(|(T_{3,2}(f)| \le \Gamma_2(x,y).\) In view of \Cref{lemma1}, we need to determine the maximum of $\Gamma_2$ over $\Delta$, given by  \eqref{delta}. We first examine at all interior points of $\Delta$. Let $(x,y)\in\Delta$ and upon partially differentiating $\Gamma_2$  with respect to $x$ and $y$, we get $(0.773,0)$, $(0.635,0)$, and $(0.48,1.17)$. Since these critical points do not belong to $\Delta$, the extremal value of $\Gamma_2$ cannot occur in the interior. Consequently, on further investigation of $\Gamma_2$ on the boundary of $\Delta$, we follow the bound for $T_{3,2}$. Sharpness is attained by $\tilde f_1$, given by \eqref{sharp iz}, for $T_{3,1}$, and by $\tilde f_2$, given by \eqref{xi}, for $T_{3,2}$.
\end{proof}

\section{\texorpdfstring{Bounds for $q$-Starlike Class $\mathcal{S}^*_{\xi_q}$}
                          {The q-Starlike Class S*xi\_q}}

In this section, it is noteworthy that as $q \to 1^-$, the results for $f \in \mathcal{S}^*_{\xi_q}$ converge smoothly to those for $f \in \mathcal{S}^*_{\xi}$, thereby confirming the consistency of the limiting approach. A distinguishing feature is the establishment of sharpness for the $q$-case using a novel extremal function method. We begin with the following sharp initial coefficient bound estimate result: 


\begin{theorem} \label{thm1}
If $f \in \mathcal{S}^*_{\xi_q}$, then
\begin{align}
    |a_2| \leq \frac{1}{q}, \quad 
    |a_3| \leq \frac{1 + q^2}{q^2(1 + q)} \quad \text{and} \quad
    |a_4| \leq \frac{6 + 12q^2 + 6q^3 + 5q^4 + 5q^5}{6q^3(1+q)(1 + q + q^2)}. \label{3.1}
\end{align}
These estimates are sharp.
\end{theorem}

\begin{proof}
Let $f \in \mathcal{S}^*_{\xi_q}$. Then by virtue of \eqref{qstar}, there exists a Schwarz function $w(z) \in \mathcal{B}_0$ such that
\begin{align}
    \frac{z\, d_q f(z)}{f(z)} = \xi_q \bigl(w(z)\bigr). \notag
\end{align}
From \eqref{1.1}, we have
\begin{align}
    \frac{z\, d_q f(z)}{f(z)} 
    &= 1 + q a_2 z + \bigl[q(1 + q)a_3 - q a_2^2\bigr] z^2  \notag \\
    &\quad + \bigl[q(1 + q + q^2)a_4 - q(q + 2)a_2 a_3 + q a_2^3\bigr] z^3 + \cdots. \label{1.4.5}
\end{align}

Using \eqref{phi_q_series}, we get
\begin{align}
    \xi_q\bigl(w(z)\bigr) 
    &= 1 + b_1 z + \bigl(b_2 + b_1^2 q\bigr) z^2 
     + \biggl(b_3 + 2b_1 b_2 q + \frac{5}{6}b_1^3 q^2\biggr) z^3 + \cdots. \label{1.4.6}
\end{align}

By comparing the coefficients in \eqref{1.4.5} and \eqref{1.4.6}, we obtain
\begin{align}
    a_2 &= \frac{b_1}{q}, \label{a_2q}\\
    a_3 &= \frac{b_2 q + b_1^2 (1 + q^2)}{q^2 (1 + q)}, \label{a_3q}\\
    a_4 &= \frac{b_3 \tau_1 + b_1 b_2 \tau_2 + b_1^3 \tau_3}{6 q^3 (1 + q)(1 + q + q^2)}, \label{a_4q}
\end{align}
where
\begin{align*}
    \tau_1 &:= 6 q^2 (1 + q), &
    \tau_2 := 6 q (2 + q + 2q^2 + 2q^3), \\
    \tau_3 &:= 6 + 12q^2 + 6q^3 + 5q^4 + 5q^5.
\end{align*}
Since $w(z)$ is rotationally invariant, we may assume without loss of generality that $b_1 \ge 0$. Furthermore, since $|b_1| \le 1$, it follows that $b_1 \in [0,1]$. From \eqref{a_2q}, we have
\[
|a_2| = \frac{|b_1|}{q} \le \frac{1}{q}.
\]

Applying ~\Cref{lemma_2} to \eqref{a_3q}, we obtain
\[
|a_3| = \left| \frac{b_1^2 (1 + q^2) + (1 - b_1^2) q \alpha}{q^2 (1 + q)} \right|
       \le \frac{1 + q^2}{q^2 (1 + q)}.
\]

Rearranging the terms in \eqref{a_4q}, we can write
\[
|a_4| = \frac{1}{q(1 + q + q^2)} \, \bigl|b_3 + \sigma b_1 b_2 + \nu b_1^3 \bigr|,
\]
where
\[
\sigma := \frac{2 + q + 2q^2 + 2q^3}{q(1 + q)}, \qquad
\nu := \frac{6 + 12q^2 + 6q^3 + 5q^4 + 5q^5}{6q^2(1 + q)}.
\]

By ~\Cref{lemma_4}, it follows that $\sigma < 4$ and $\nu > \frac{1}{12}(\sigma^2 + 8)$ for $q \in (0,1)$. Hence,
\[
|a_4| \le \frac{6 + 12q^2 + 6q^3 + 5q^4 + 5q^5}{6 q^3 (1 + q)(1 + q + q^2)}.
\]

Thus, using \eqref{df'}, we verify that the bounds in \eqref{3.1} are sharp, since equality is attained for the extremal function 
$f_1 : \mathbb{D} \to \mathbb{C}$, given by , 
\begin{align}
     f(z) * \frac{z}{(1 - qz)(1 - z)} = f(z) \cdot \xi_q(z).  \label{sharp}
\end{align}
This completes the proof.
\end{proof}

We now proceed to estimate the Fekete-Szeg\"o bound:
\begin{theorem} \label{thm2}
Let $f \in \mathcal{S}^*_{\xi_q}$, then
\[
|a_3 - \mu a_2^2| \leq \frac{1}{q(1+q)} \max\left\{ 1, \left| \mu(1+q) - (1+q^2) \right| \right\},\quad \mu \in \mathbb{C}.
\label{3.18}
\]
\end{theorem}

\begin{proof}
Let $f \in \mathcal{S}^*_{\xi_q}$. Using \eqref{a_2q} and \eqref{a_3q}, we have
\begin{align}
    |a_3 - \mu a_2^2|
    &= \left| \frac{b_2 q + b_1^2 (1 + q^2)}{q^2 (1 + q)} - \mu\frac{b_1^2}{q^2} \right|.
    \label{2.4.15}
\end{align}

By expressing \eqref{2.4.15} in terms of the coefficients $c_i$ ($i=1,2$) using \eqref{2.4.17} and subsequently applying \Cref{lemma_3}, we obtain

\begin{align*}
    |a_3 - \mu a_2^2|
    &= \left| \frac{2c_2 q + (1 - q + q^2)c_1^2}{4q^2(1+q)} 
    - \mu \frac{c_1^2}{4q^2} \right| \\
    &\leq \frac{1}{2q(1+q)} 
    \left| c_2 - 
    \left( \frac{\mu(1+q) - 1 + q - q^2}{2q} \right) c_1^2 \right| \\
    &\leq \frac{1}{q(1+q)} 
    \max \left\{ 
        1, 
        \left| \mu(1+q) - (1+q^2) \right| 
    \right\}, 
    \qquad \mu \in \mathbb{C}.
\end{align*}

Hence, the desired inequality follows.
\end{proof}

By setting $\mu = 1$ in \Cref{thm2}, we obtain the following sharp result:
\begin{corollary} \label{corq}
Let $f \in \mathcal{S}_{\xi_q}^*$, then
\begin{align}
    |a_3 - a_2^2| \leq \frac{1}{q(1+q)}. \notag
\end{align}
    Above inequality is sharp due to the function $f_2 : \mathbb{D} \to \mathbb{C}$, given by 
\begin{align}
    f_2(z):=f(z) * \frac{z}{(1-qz)(1-z)} = f(z) \cdot \xi_q(z^2). \label{sharp1}
\end{align}
\end{corollary}

\noindent
Note that, if $f \in \mathcal{S}^*_{\xi_q}$, then the second Hankel determinant satisfies
\begin{align*}
    |\mathcal{H}_{2,1}(f)| = |a_1 a_3 - a_2^2| \leq \frac{1}{q(1+q)}, \quad \text{where } a_1 = 1.
\end{align*}

We now obtain the sharp bound for the second order Hankel determinant:
\begin{theorem}\label{thm3}
     If $f \in \mathcal{S}^*_{\xi_q}$, then
    \begin{align}
        |\mathcal{H}_{2,2}(f)| \leq \frac{1}{q^2 (1+q)^2}. \label{h22}
    \end{align}
The estimate is sharp.
\end{theorem}

\begin{proof}
Let $f \in \mathcal{S}^*_{\xi_q}$, then from \eqref{Hankel_det} and \eqref{a_2q}-\eqref{a_4q}, we have
\begin{align}
    |\mathcal{H}_{2,2}(f)|
    &= |a_2 a_4 - a_3^2|
    = \left|
        \frac{
            b_1 b_3 \tau_4 - b_2^2 \tau_5 + b_1^2 b_2 \tau_6 - b_1^4 \tau_7
        }{
            6 q^2 (1+q)^2 (1 + q + q^2)
        }
      \right|,
    \label{h2.21}
\end{align}
where
\begin{align*}
    \tau_4 &= 6(1+q)^2, &   \tau_5 &= 6(1+q+q^2),\\
  \tau_6& = 6(1 - q + 2q^2),& \tau_7 &= 6 - 6q + 7q^2 - 4q^3 + q^4.
\end{align*}

Using ~\Cref{lemma_2}, \eqref{h2.21} reduces to
\begin{align}
    |\mathcal{H}_{2,2}(f)|
    &= \tfrac{
        |-b_1^4 \tau_7
        + b_1^2 (1 - b_1^2) \tau_6 \alpha
        - (1 - b_1^2)^2 \tau_{5} \alpha^2
        + b_1 (1 - b_1^2) \tau_{4} (\beta( 1 - |\alpha|^2)-b_1 \alpha^2 )|
      }{
        6 q^2 (1+q)^2 (1+q+q^2)
      },
    \label{h2.22}
\end{align}
For $b_1 \in \{0,1\}$, \eqref{h2.22} simplifies to
\begin{align}
    |\mathcal{H}_{2,2}(f)| =
    \begin{cases}
        \displaystyle
        \frac{|\alpha|^2}{q^2(1+q)^2}
        \le \frac{1}{q^2(1+q)^2},
        & b_1 = 0,\; |\alpha| \le 1,\\[8pt]\\
        \displaystyle
        \frac{6 - 6q + 7q^2 - 4q^3 + q^4}{
            6 q^2 (1+q)^2 (1+q+q^2)
        },
        & b_1 = 1.
    \end{cases}
    \label{4.17}
\end{align}

For $b_1 \in (0,1)$, applying the triangle inequality to \eqref{h2.22} and using $|\beta| \le 1$, we obtain
\begin{align}
    |\mathcal{H}_{2,2}(f)|
    \le \frac{b_1(1 - b_1^2)}{q(1 + q + q^2)} \, Y(A,B,C),
    \label{h2.23}
\end{align}
where
\[
    Y(A,B,C) = \big(|A + B\alpha + C\alpha^2| + 1 - |\alpha|^2\big),
\]
and
\begin{align}
    A = -\frac{b_1^3 \tau_7}{(1 - b_1^2) \tau_4}, \quad
    B = \frac{b_1 \tau_6}{(1+q)^2}, \quad
    C = -\frac{(1 + q + b_1^2 q + q^2)}{b_1(1+q)^2}.\nonumber
    \label{h2.24}
\end{align}

From ~\Cref{lemma5}, we obtain
\begin{align*}
   \varphi_1(b_1,q):= AC &= \frac{
        b_1^2 (1 + q + b_1^2 q + q^2)
        (6 - 6q + 7q^2 - 4q^3 + q^4)
      }{
        6 (1 - b_1^2)(1+q)^4
      } \ge 0 \end{align*} 
and
\begin{align*}
    \varphi_2(b_1,q)&:=  |B| - 2(1 - |C|)
    \\&= \frac{
        -2b_1(1+q)^2 + 2(1+q+q^2) + b_1^2(1 + q + 2q^2)
      }{
        b_1(1+q)^2
      } \ge 0,
\end{align*}
which is evident from \Cref{fig3} and \Cref{fig4}.
\begin{figure}[ht!]
\centering
\begin{minipage}{0.48\textwidth}
    \centering
    \includegraphics[width=0.9\textwidth]{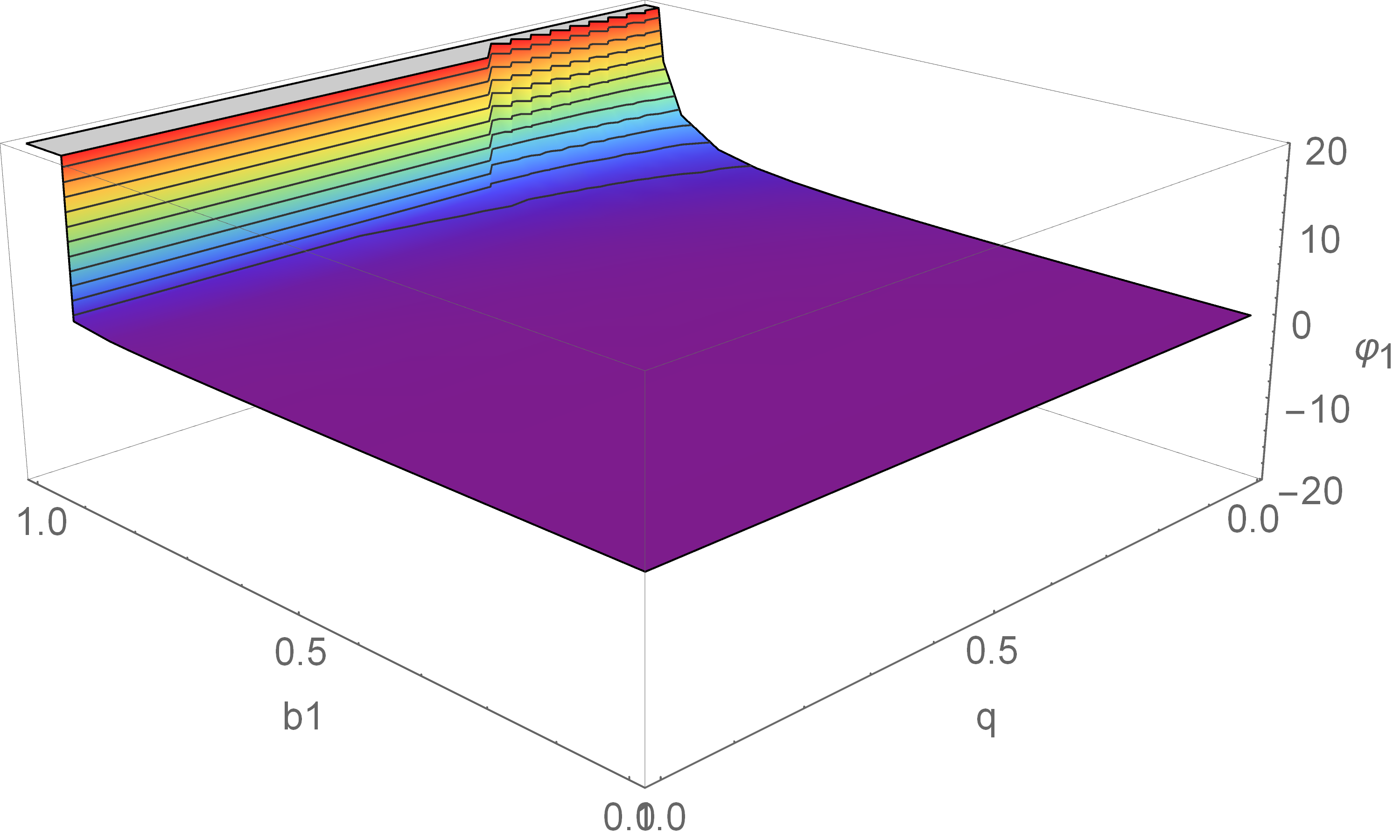}
    \caption{Plot of $\varphi_1(b_1, q)$ for $b_1, q \in (0,1)$.}
    \label{fig3}
\end{minipage}
\hfill
\begin{minipage}{0.48\textwidth}
    \centering
    \includegraphics[width=0.9\textwidth]{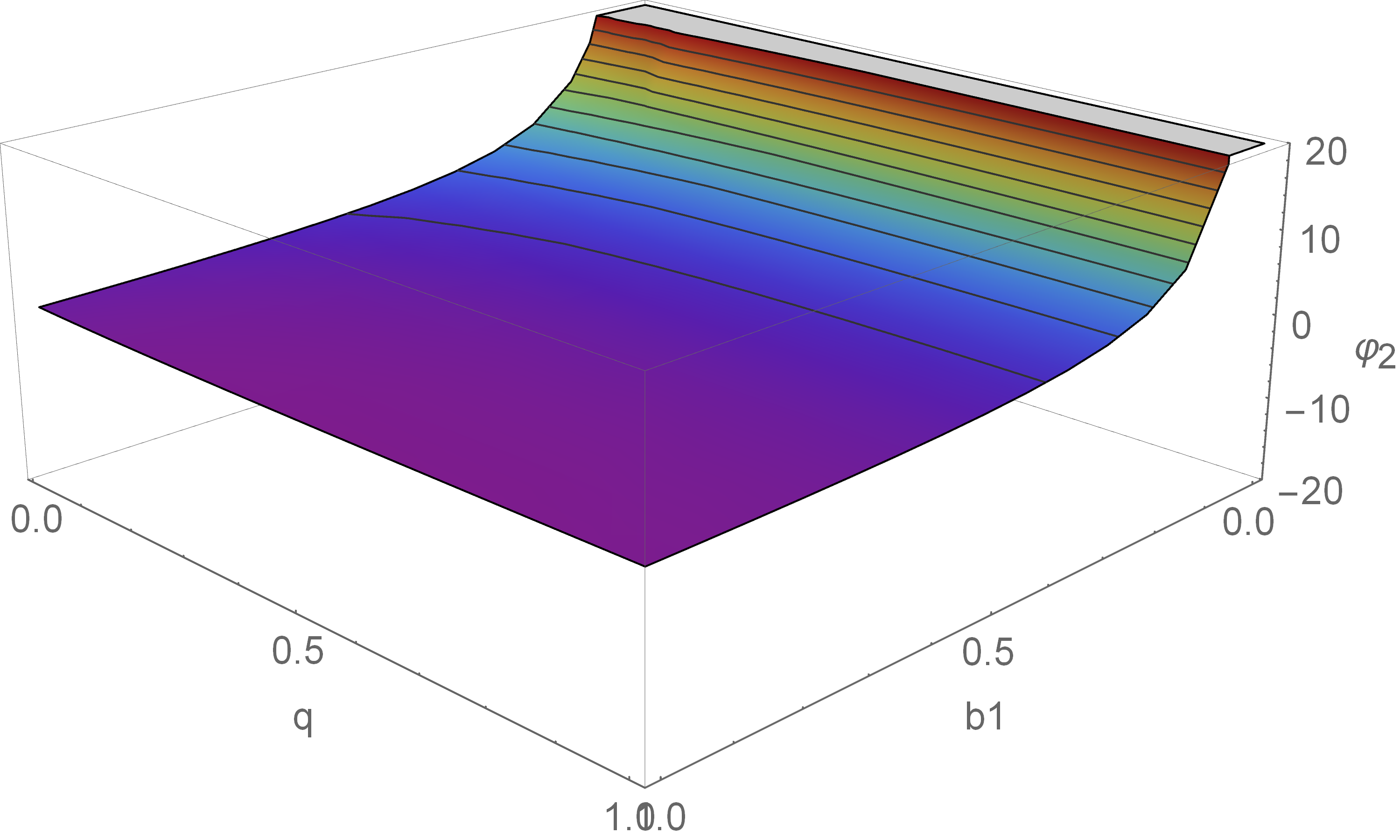}
    \caption{Plot of $\varphi_2(b_1, q)$ for $b_1, q \in (0,1)$.}
    \label{fig4}
\end{minipage}
\end{figure}

Thus, $|B| \ge 2(1 - |C|)$ for $q\in(0,1)$, which implies that
\[
    Y(A,B,C) = |A| + |B| + |C|.
\]
Therefore, \eqref{h2.23} simplifies to
\begin{align}
    |\mathcal{H}_{2,2}(f)|
    &\le \frac{b_1(1 - b_1^2)}{q(1+q+q^2)} (|A| + |B| + |C|)
    < \frac{1}{q^2(1+q)^2}.
    \label{4.20}
\end{align}

Combining \eqref{4.17} and \eqref{4.20} yields \eqref{h22}. 
Furthermore, the estimate is sharp, and equality holds for the extremal function $f_2$ defined in \eqref{sharp1}.
\end{proof}
We now proceed to the corresponding Toeplitz determinant bounds, beginning with the second-order Toeplitz determinant:
\begin{theorem} \label{thm4}
    If $f \in \mathcal{S}^*_{\xi_q}$, then
    \begin{align}
        |\mathcal{T}_{2,1}(f)| \leq 1+\frac{1}{q^2}, \qquad
        |\mathcal{T}_{2,2}(f)| \leq \frac{1}{q^2}+\frac{(1+q^2)^2}{q^4(1+q)^2} \qquad \text{and} \qquad |\mathcal{T}_{2,3}(f)| \leq \frac{1}{q^2 (1 + q)^2 }.\notag
    \end{align}
These estimates are sharp.

\end{theorem}

\begin{proof}
Let $f \in \mathcal{S}^*_{\xi_q}$. By substituting the values of $a_2$ and $a_3$ from \eqref{a_2q} and \eqref{a_3q} into 
\(
\mathcal{T}_{2,2} = a_2^2 - a_3^2,
\)
we obtain
\begin{equation}
    |\mathcal{T}_{2,2}(f)|
    = \left|
        \frac{b_1^2}{q^2}
        - \frac{(b_2 q + b_1^2 (1 + q^2))^2}{q^4 (1 + q)^2}
      \right|.
    \label{8.4.36}
\end{equation}
Further on applying the ~\Cref{lemma_2}, ~\eqref{8.4.36} reduces to
\begin{align*}
    |\mathcal{T}_{2,2}(f)|
    &\le\left|\frac{b_1^2}{q^2}
       + \frac{(q \alpha+ b_1^2 (1 - q + q^2\alpha))^2}{q^4 (1 + q)^2}\right| = \frac{1}{q^2}+\frac{(1+q^2)^2}{q^4(1+q)^2}.
\end{align*}

Now, to estimate the bound for $\mathcal{T}_{2,3}$, using \eqref{toeplitz_det}, \eqref{a_3q}, and \eqref{a_4q}, we obtain
\begin{align}
|\mathcal{T}_{2,3}(f)| &= |a_3^2 - a_4^2|
= \left| \frac{\Omega_1 - \Omega_2}{36 q^6 (1 + q)^2} \right|,
\label{t23.2}
\end{align}
where
\begin{align*}
\Omega_1:= 36q^2(b_2q+b_1^2(1+q^2))^2,\quad
\Omega_2:= \frac{(6b_3q^2\tau_{8}+b_1b_2\tau_{9}+b_1^3\tau_{10})^2}{\tau_{11}^2}
\end{align*}
and
\begin{align*}
\tau_{8}& := 1+q, & \tau_{9}&:=6 (2 + q (1 + 2 q (1 + q))),\\ \tau_{10}&:=6 + q^2 (12 + q (6 + 5 q (1 + q))), & \tau_{11}&:=1 + q + q^2.
\end{align*}

Using ~\Cref{lemma1}, \eqref{t23.2} reduces to
\begin{align*}
|\mathcal{T}_{2,3}(f)| = 
\frac{\Omega_3 - \Omega_4}{36 q^6 (1 + q)^2 },
\end{align*}
where
\begin{align*}
\Omega_3 &:= 36 q^2 (q + |b_1|^2 (1 - q + q^2))^2,\\
\Omega_4 &:=\displaystyle\frac{1}{\tau_{11}^2}\left(1-|b_1|^2-\frac{|b_2|^2}{1+|b_1|}\right)\tau_1+|b_1|(1-|b_1|^2)\tau_2+|b_1|^3\tau_{10}.
\end{align*}

Setting $x := |b_1|$ and $y := |b_2|$, we obtain $|\mathcal{T}_{2,3}(f)| \leq \Gamma(x, y)$, where
\begin{align*}
\Gamma(x, y) = \frac{\Omega_5 - \Omega_6}{36 q^6 (1 + q)^2},
\end{align*}
with
\begin{align*}
\Omega_5 &:= 36 q^2 (q + x^2 (1 - q + q^2))^2,\\
\Omega_6 &:=\displaystyle\frac{1}{\tau_{11}^2}\left(1-x^2-\frac{y^2}{1+x}\right)\tau_1+x(1-x^2)\tau_2+x^3\tau_{10}.
\end{align*}

By ~\Cref{lemma1}, we seek the maximum of $\Gamma$ over $\Delta$, given by \eqref{delta}). Initially, we consider the interior points of $\Gamma$. By considering $\partial \Gamma / \partial y = 0$, gives
\begin{align*}
    y=y_0:=\frac{\sqrt{\tau_{12}}}{q\sqrt{6(1+q)}}
\end{align*}
where
\begin{align*}
    \tau_{12}&:=6 x^3 (1 + x) + 5 q^5 x^3 (1 + x) - 12 q (-1 + x) x (1 + x)^2\\& +  q^4 x (12 + 12 x - 7 x^2 - 7 x^3) + 6 q^2 (1 + 2 x + x^4) \\&-  6 q^3 (-1 - 3 x - x^2 + 2 x^3 + x^4)
\end{align*}
For the existence of $y_0$, it should belong to $(0,1)$. However, in further estimation, we observe that there does not exist any $x \in (0,1)$. So, we find no critical points $(x_0, y_0)$ in the interior of $\Delta$. Thus, $\Gamma$ achieves its maximum at the boundary of $\Delta$. On the boundary, we have 
\begin{align*}
\Gamma(x, 0) &\leq 
\frac{\mathcal{M}}{36 q^6 (1 + q)(1 + q + q^2)^2}, \quad (0 \le x \le 1),\\
\Gamma(0, y) &\le \frac{1}{q^2 (1 + q)^2}, \quad (0 \le y \le 1),\\
\Gamma(x, 1 - x^2) &\le
\frac{\mathcal{M}}
{36 q^6 (1 + q)(1 + q + q^2)^2}, \quad (0 \le x \le 1),
\end{align*}
where $\mathcal{M}=36 - 36q + 144q^2 - 144q^3 + 168q^4 - 180q^5 + 48q^6 - 84q^7 - 11q^8 - 11q^9$. 
The bounds for $\mathcal{T}_{2,2}$ and $\mathcal{T}_{2,3}$ are sharp, equality is attained by the extremal function $f_3 : \mathbb{D} \to \mathbb{C}$, given by
\begin{align}
   f_3(z):= f(z) * \frac{z}{(1 - qz)(1 - z)} = f(z) \cdot \xi_q(iz).  \label{sharp iqz}
\end{align}
Moreover, the sharpness for $\mathcal{T}_{2,3}$ is also realized by the function $f_2$, given in \eqref{sharp1}.
\end{proof}

Next, we establish the third-order Toeplitz determinant using the methodology described earlier:
\begin{theorem} \label{thm7}
     If $f \in \mathcal{S}^*_{\xi_q}$, then
    \begin{align}
        |\mathcal{T}_{3,1}(f)| &\leq \frac{(1 - q)^4 (1 + 4 q + 5 q^2 + 4 q^3 + q^4)}{q^2(1+q)^2}\qquad \text{and}\label{t31}\\ 
        |\mathcal{T}_{3,2}(f)|& \leq \frac{\mathcal{M}_1(6 + 6 q^2 - 6 q^3 - 7 q^4 - q^5)}{36 q^9 (1 + q)^4 (1 + q + q^2)^2},\label{t32}
    \end{align}
where \(\mathcal{M}_1=12 + 12 q + 42 q^2 + 24 q^3 + 48 q^4 - 18 q^5 - 23 q^6 - 51 q^7 - 27 q^8 - 11 q^9\).  The estimate is sharp, with sharpness verified by the extremal function given in \eqref{sharp iqz} for \eqref{t31} and in \eqref{sharp} for \eqref{t32}, respectively.
\end{theorem}

\begin{remark}
It is observed that the  sharp bounds for the inverse coefficients of $g(w)=f^{-1}(w)$ where, $f\in\mathcal{S}^*_{\xi}$ and $\mathcal{S}^*_{\xi_q}$, coincide precisely with all the discussed coefficient problems for the class. The coincidence is not merely 
numerical in each case, sharpness is achieved by the same extremal functions, indicating 
that the extremal configuration of the class is preserved under inversion.
\end{remark}

\begin{conjecture}
Let $f \in \mathcal{S}^{*}_{\xi}$ and let $d$ and  denote the inverse coefficients of $f$. Then for all $n \geq 5$, the sharp bounds for $|d_n|$
 coincide with those for $|a_n|$. Moreover, the sharp estimates for the associated functionals formed from $d_n$ - coincide precisely with the established functionals within the class $\mathcal{S}^{*}_{\xi}$.
\end{conjecture}

\section*{Conclusion}
This study introduces two novel subclasses of starlike functions, $\mathcal{S}^*_{\xi}$ 
and $\mathcal{S}^*_{\xi_q}$, where $\mathcal{S}^*_{\xi}$ arises as the limiting case of 
$\mathcal{S}^*_{\xi_q}$ as $q \to 1^-$. Defined via subordination and $q$-calculus, the 
class $\mathcal{S}^*_{\xi_q}$ provides a natural $q$-analogue of the classical framework, 
for which sharp coefficient estimates are derived for the initial coefficients and all 
related determinants and functionals. Beyond their theoretical interest, the functionals studied here carry meaningful 
implications across several applied domains. In fluid dynamics, Fekete-Szeg\"{o} 
bounds ensure controlled geometric distortion in conformal mappings used for 
two-dimensional potential flow modeling. In signal processing and control 
theory, coefficient bounds guarantee stability and robustness of analytic transfer 
functions. In image processing and medical imaging, these estimates enable smooth, 
non-overlapping transformations and provide quantitative measures of shape complexity. 
Within quantum calculus itself, the $q$-starlike framework ensures structural stability 
of $q$-analogue formulations arising in discrete and quantum models.

\section*{Declarations}
\subsection{Author Contributions.}
All authors contributed equally to this work.

\subsection{Funding.}
Not applicable.
\bibliographystyle{plain}
\bibliography{refer}

\end{document}